\documentclass[10pt]{amsart}
\usepackage[english]{babel}
\usepackage{tikz}\usepackage{tikz-cd}

\usetikzlibrary{patterns}
\usetikzlibrary{decorations.pathreplacing}

\usepackage[DIV10, headinclude, BCOR 10mm]{typearea}
\usepackage{amssymb}
\usepackage[initials]{amsrefs}
\usepackage{mathrsfs}
\usepackage{enumitem}
\usepackage{comment}

\usepackage[all]{xy}
\SelectTips{cm}{}

\newcommand{\bbN}{{\mathbb N}}
\newcommand{\bbQ}{{\mathbb Q}}

\newcommand{\bbC}{{\mathbb C}}

\newcommand{\calA}{\mathcal{A}}
\newcommand{\calB}{\mathcal{B}}

\newcommand{\calF}{\mathcal{F}}

\newcommand{\CW}{{\mathrm{cw}}}

\newcommand{\bs}{\backslash}

\newcommand{\id}{\operatorname{id}}
\newcommand{\im}{\operatorname{im}}

\newcommand{\coker}{\operatorname{coker}}

\newcommand{\trace}{\operatorname{tr}}
\newcommand{\betti}{\operatorname{\beta}}
\newcommand{\vol}{\operatorname{vol}}

\newcommand{\closure}{\operatorname{clos}}
\newcommand{\hs}{\operatorname{hs}}
\newcommand{\map}{\operatorname{map}}

\newcommand{\norm}[1]{{\lVert #1\rVert}}

\newtheorem{theorem}{Theorem}[section]

\newtheorem{corollary}[theorem]{Corollary}

\newtheorem{prop}[theorem]{Proposition}
\theoremstyle{definition}
\newtheorem{definition}[theorem]{Definition}
\newtheorem{defn}[theorem]{Definition}

\newtheorem{example}[theorem]{Example}

\newtheorem{remark}[theorem]{Remark}

\newtheorem{notation}[theorem]{Notation}

\numberwithin{equation}{section}

\begin{document}
\title{$\ell^2$-Betti numbers of discrete and non-discrete groups}
\author{Roman Sauer}
\address{Karlsruhe Institute of Technology, Karlsruhe, Germany}
\email{roman.sauer@kit.edu}
\thanks{I thank Henrik Petersen for corrections and helpful comments.}

\maketitle
\begin{abstract}
	In this survey we explain the definition of $\ell^2$-Betti numbers of locally compact groups -- both
	discrete and non-discrete. Specific topics include proportionality principle, L\"uck's dimension function, Petersen's definition for locally compact groups, some concrete examples.
\end{abstract}

\section{Introduction}
This short survey addresses readers who, motivated by geometric group theory, seek after a
brief overview of the (algebraic) definitions of $\ell^2$-Betti numbers
of discrete and non-discrete locally compact groups.
I limit myself to the group case, mostly ignoring the theory
of $\ell^2$-Betti numbers of equivariant spaces.

$\ell^2$-Betti numbers share many formal properties with ordinary Betti numbers, like
K\"unneth and Euler-Poincare formulas. But the powerful \emph{proportionality principle}
is a distinctive feature of $\ell^2$-Betti numbers. It also was a motivation to generalise the
theory of $\ell^2$-Betti numbers to locally compact groups.

In its easiest form, the proportionality principle states that
the $\ell^2$-Betti number in degree~$p$, which is henceforth denoted as $\betti_p(\_)$,
of a discrete group $\Gamma$ and
a subgroup $\Lambda<\Gamma$ of finite index satisfy the relation
\begin{equation}\label{eq: prop for finite index subgroups}
	\betti_p(\Lambda)=[\Gamma:\Lambda]\betti_p(\Gamma).
\end{equation}
The next instance of the proportionality principle involves
lattices $\Gamma$ and $\Lambda$ in a semisimple Lie group~$G$ endowed
with Haar measure~$\mu$.
It says that their
$\ell^2$-Betti numbers scale according to their
covolume:
\begin{equation}\label{eq: prop for lattices}
	\betti_p(\Gamma)\mu(\Lambda\bs G)=\betti_p(\Lambda)\mu(\Gamma\bs G)
\end{equation}
The proof of~\eqref{eq: prop for lattices} becomes easy
when we use the original, analytic definition of $\ell^2$-Betti numbers
by Atiyah~\cite{atiyah}.
According to this definition, the $p$-th $\ell^2$-Betti numbers of
$\Gamma$, denoted by $\betti_p(\Gamma)$, is given in terms of the heat kernel
of the Laplace operator on $p$-forms on the symmetric space $X=G/K$:
\[
	\betti_p(\Gamma)=\lim_{t\to\infty}\int_{\calF}\trace_\bbC(e^{-t\Delta^p}(x,x))\operatorname{dvol}.
\]
Here $\calF\subset X$ is a measurable fundamental domain for the $\Gamma$-action on $X$
and $e^{-t\Delta^p}(x,x): \operatorname{Alt}^p(T_xX)\to \operatorname{Alt}^p(T_xX)$ is the integral kernel -- called \emph{heat kernel} -- of the bounded operator $e^{-t\Delta^p}$ obtained from the unbounded
Laplace operator $\Delta^p$ on $L^2\Omega^p(X)$ by spectral calculus. Since $G$ acts
transitively on $X$ by isometries, it is clear that the integrand in the above formula
is constant in~$x$. So there is a constant $c>0$ only depending
on $G$ such that $\betti_p(\Gamma)=c\vol(\Gamma\bs X)$, from which one
deduces~\eqref{eq: prop for lattices}.

Gaboriau's theory of $\ell^2$-Betti numbers of measured equivalence
relations~\cite{gaboriau} (see also~\cite{sauer-betti} for an extension to measured groupoids) greatly generalised~\eqref{eq: prop for lattices}
to the setting of measure equivalence.

\begin{defn}
	Let $\Gamma$ and $\Lambda$ be countable discrete groups.
	If there is a Lebesgue measure
	space $(\Omega, \nu)$ on which $\Gamma$ and $\Lambda$ act
	in a commuting, measure preserving way such that both
	the $\Gamma$- and the $\Lambda$-action admit measurable fundamental domains of
	finite $\nu$-measure, we call $\Gamma$ and $\Lambda$ \emph{measure equivalent}.
	We say that $(\Omega,\nu)$ is a \emph{measure coupling} of $\Gamma$
	and $\Lambda$.
\end{defn}

\begin{theorem}[Gaboriau]
	If $(\Omega,\nu)$ is a measure coupling of $\Gamma$ and $\Lambda$, then
	\[
		\betti_p(\Gamma)\nu(\Lambda\bs\Omega)=\betti_p(\Lambda)\nu(\Gamma\bs\Omega).
	\]
\end{theorem}

If $\Gamma$ and $\Lambda$ are lattices in the same locally compact $G$, then
$G$ endowed with its Haar measure is a measure coupling for the actions described
by $\gamma\cdot g:=\gamma g$ and $\lambda\cdot g:=g\lambda^{-1}$ for $g\in G$ and
$\gamma\in \Gamma$, $\lambda\in \Lambda$. So Gaboriau's theorem
generalises~\eqref{eq: prop for lattices} to lattices in arbitrary locally
compact groups.

The geometric analogue of measure equivalence is quasi-isometry. 
One may wonder whether $\ell^2$-Betti numbers of quasi-isometric groups are also 
proportional. This is not true!  
The groups $\Gamma=F_3* (F_3\times F_3)$ and $\Lambda=F_4*(F_3\times F_3)$ are quasi-isometric. Their Euler characteristics are $\chi(\Gamma)=1$ and $\chi(\Lambda)=0$. By the Euler-Poincare formula (see Theorem~\ref{thm: euler poincare}) the $\ell^2$-Betti numbers of $\Gamma$ and $\Lambda$ are not proportional. We refer to~\citelist{\cite{harpe}*{p.~106}\cite{gaboriau-survey}*{2.3}} for a discussion of this example. However, the vanishing of $\ell^2$-Betti numbers is a quasi-isometry invariant. This is discussed in Section~\ref{sub: qi}. 

Motivated by~\eqref{eq: prop for finite index subgroups}
and~\eqref{eq: prop for lattices}, Petersen introduced
$\ell^2$-Betti numbers $\betti_\ast(G,\mu)$ of an arbitrary second countable, locally compact unimodular
group~$G$ endowed
with a Haar measure~$\mu$. The first instances of $\ell^2$-invariants for non-discrete locally compact groups 
are Gaboriau's first $\ell^2$-Betti number of a unimodular graph~\cite{gaboriau-percolation}, which is essentially one of its automorphism group, and the $\ell^2$-Betti numbers of buildings in the work of  
Dymara~\cite{dymara} and Davis-Dymara-Januszkiewicz-Okun~\cites{coxeter}, which are essentially ones of the automorphism groups of the buildings. 
We refer to~\cite{petersen+sauer+thom} for more information on the relation of the latter to Petersen's definition.

If $G$ is discrete, we always take $\mu$ to be
the counting measure. If a locally compact group possesses a lattice, it is unimodular.
Petersen~\cite{petersen-phd} (if $\Gamma<G$ is cocompact or $G$ is totally disconnected)
and then Kyed-Petersen-Vaes~\cite{kyed+petersen+vaes} (in general)
showed the following generalisation of~\eqref{eq: prop for finite index subgroups}. A proof which is much easier but only works for
totally disconnected groups can be found in~\cite{petersen+sauer+thom}.

\begin{theorem}[Petersen, Kyed-Petersen-Vaes]\label{thm: locally compact prop principle}
	Let $\Gamma$ be a lattice in a second countable, locally compact group $G$ with Haar measure~$\mu$. Then
	\[
		\betti_p(\Gamma)=\mu(\Gamma\bs G)\betti_p(G,\mu).
	\]
\end{theorem}

\subsection*{Structure of the paper}
In Section~\ref{sec: cohomology} we explain the definition of continuous cohomology. And we explain
how to compute it for discrete and totally disconnected groups by geometric models.
In Section~\ref{sec: dimension} we start with the definition of the von Neumann algebra
of a unimodular locally compact group and its semifinite trace. The goal is to understand
L\"uck's dimension function
for modules over von-Neumann algebras. In Section~\ref{sec: betti numbers} we define $\ell^2$-Betti numbers,
comment on their quasi-isometry invariance, and present some computations.

\section{Continuous cohomology}\label{sec: cohomology}

Throughout, let $G$ be a second countable, locally compact group, and let
$E$ be topological vector space with a continuous $G$-action, i.e.~the action
map $G\times E\to E$, $(g,e)\mapsto ge$, is continuous. A topological vector space
with such a $G$-action is called a \emph{$G$-module}. A \emph{$G$-morphism} of
$G$-modules is a continuous, linear, $G$-equivariant map.

\subsection{Definition via bar resolution} 
\label{sub:definition_cohomology}

Let $C(G^{n+1}, E)$ be the vector space of continuous maps from $G^{n+1}$ to~$E$.
The group $G$ acts from the left on $C(G^{n+1}, E)$ via
\[
	(g\cdot f)(g_0,\ldots, g_n)=gf(g^{-1}g_0,\ldots, g^{-1}g_n).
\]
If $G$ is discrete, the continuity requirement is void, and $C(G^{n+1},E)$
is the vector space of all maps $G^{n+1}\to E$.
The fixed set $C(G^{n+1}, E)^G$ is just the set of continuous equivariant maps.
The \emph{homogeneous bar resolution} is the chain complex
\[
	C(G, E)\xrightarrow{d^0} C(G^2, E)\xrightarrow{d^1} C(G^3, E)\to\ldots
\]
with differential
\[
d^n(f)(g_0,\ldots, g_{n+1})=\sum_{i=0}^{n+1}(-1)^if(g_0,\ldots, \hat{g_i},\ldots, g_{n+1}).
\]
The chain groups $C(G^{n+1}, E)$ are endowed with the compact-open topology
turning them into $G$-modules.  This is
a non-trivial topology even for discrete $G$, where it coincides with pointwise
convergence.

Since the differentials are $G$-equivariant we obtain a chain complex $C(G^{\ast+1}, E)^G$ by restricting to the equivariant maps.

\begin{defn}
	The cohomology of $C(G^{\ast+1}, E)^G$ is called the \emph{continuous cohomology}
	of $G$ in the $G$-module $E$ and denoted by $H^\ast(G,E)$.
\end{defn}

The differentials are continuous but usually have non-closed image, which leads to a non-Hausdorff quotient topology on the continuous cohomology. Hence it is natural
to consider the following.

\begin{defn}
	The \emph{reduced continuous cohomology} $\bar H^\ast(G,E)$
	of $G$ in $E$ is defined
	as the quotient $\ker(d^n)/\closure(\im(d^{n-1}))$ of $H^n(G, E)$,
	where we take the quotient by the closure of the image of the differential.
\end{defn}

In homological algebra it is common to compute
derived functors, such as group cohomology,
by arbitrary injective resolutions. The specific definition of continuous
cohomology by
the homogeneous bar resolution,
which is nothing else than usual group cohomology if
$G$ is discrete, is the quickest definition in the topological setting.
But there is also an approach in the sense of homological algebra, commonly referred to
as \emph{relative homological algebra}.

We call an injective $G$-morphism of $G$-modules \emph{admissible} if it admits
a linear, continuous (not necessarily $G$-equivariant) inverse.

\begin{defn}
	A $G$-module $E$ is \emph{relatively injective} if
	for any admissible injective $G$-morphism $j: U\to V$ and a $G$-morphism
	$f: U\to E$ there is a $G$-morphism $\bar f: V\to E$ such that $\bar f\circ j=f$.
\end{defn}

\begin{example}\label{exa: relatively injective modules}
	Let $E$ be a $G$-module. Then $C(G^{n+1}, E)$ is relatively injective.
	Let $j: U\to V$ be an admissible $G$-morphism. Let $s: V\to U$ be a linear continuous map with $s\circ j=\id_U$.
	Given a $G$-morphism $f: U\to C(G^{n+1}, E)$, the $G$-morphism
	\[
		(\bar f)(v)(g_0,\ldots g_n)=f\bigl(g_0s(g_0^{-1}v)\bigr)(g_0,\ldots, g_n)
	\]
	satisfies $\bar f\circ j=f$.
	Similarly, if $K$ is a compact subgroup, then $C((G/K)^{n+1}, E)$ is relatively
	injective. Here the extension is given by
	\[
		(\bar f)(v)([g_0],\ldots [g_n])=\int_K f\bigl(g_0ks(k^{-1}g_0^{-1}v)\bigr)([g_0],\ldots, [g_n])d\mu(k),
	\]
    where $\mu$ is the Haar measure normalised with $\mu(K)=1$.
\end{example}

As the analog of the fundamental lemma of homological algebra we have:

\begin{theorem}\label{thm: fundamental lemma}
	Let $E$ be a $G$-module. Let $0\to E\to E^0\to E^1\to\ldots$
	be a resolution of $E$ by relatively injective $G$-modules $E^i$.
	Then the cohomology of $(E^\ast)^G$ is (topologically) isomorphic
	to the continuous cohomology of $G$ in~$E$.
\end{theorem}

\subsection{Injective resolutions to compute continuous cohomology} 
\label{sub:how_to_compute}

The homogeneous bar resolution is useful for proving general properties of
continuous cohomology. But other
injective resolutions coming from geometry are better suited for computations.

\begin{defn}\label{def: geometric model}
	Let $G$ be totally disconnected (e.g.~discrete).
	Let $X$ be a cellular complex
	on which $G$ acts cellularly and continuously. We require
	that for each open cell
$e\subset X$ and each $g\in G$ with $ge\cap e\ne\emptyset$ multiplication by $g$
is the identity on~$e$.
	We say that $X$ is a \emph{geometric model} of~$G$ if $X$ is contractible, its $G$-stabiliser
	are open and compact, and the $G$-action on the $n$-skeleton $X^{(n)}$ is cocompact
	for every $n\in\bbN$.
\end{defn}

For simplicial $G$-actions the requirement on open cells
can always be achieved by passage to the barycentric subdivision.
A cellular complex with cellular $G$-action that satisfies the above
requirement on open cells and whose stabilisers are open is a \emph{$G$-CW-complex} in
the sense of~\cite[II.1.]{tomdieck}.
This means that the
$n$-skeleton $X^{(n)}$ is built from $X^{(n-1)}$ by attaching $G$-orbits of
$n$-cells according to pushouts of $G$-spaces of the form:
\[
	\xymatrix{
	\bigsqcup_{U\in \calF_n}G/U\times S^{n-1}\ar@{^{(}->}[d]\ar[r] & X^{(n-1)}\ar@{^{(}->}[d]\\
	\bigsqcup_{U\in \calF_n}G/U\times D^n\ar[r]     & X^{(n)}
	}
\]
Here $\calF_n$ is a set of representatives of
conjugacy classes
of stabilizers of $n$-cells. We require that each $X^{(n)}$ is cocompact. So $\calF_n$ is finite.
Each coset space $G/U$ is discrete.
Let us fix a choice of pushouts, which is not part of the data of a cellular complex and
corresponds to an equivariant choice of orientations for the cells.
The horizontal maps induces an isomorphism in relative homology by excision.
Let $C_\ast^{\CW}(X)$ be the cellular chain complex with $\bbC$-coefficients:
We obtain isomorphisms of discrete $G$-modules:
\begin{equation}\label{eq: geometric chain group}
	\bigoplus_{U\in \calF_n}\bbC[G/U]\cong H_n(\bigsqcup_{U\in \calF_n} G/U\times (D^n, S^{n-1}))\xrightarrow{\cong} H_n(X^{(n)}, X^{(n-1)})\overset{\mathrm{def}}{=}C_n^{\CW}(X)
\end{equation}
The $G$-action $(gf)(x)=gf(g^{-1}x)$ turns
\[\hom_\bbC(C_n^{\CW}(X),E)\cong\prod_{U\in\calF_n}C(G/U, E)= \bigoplus_{U\in\calF_n}C(G/U, E)\] into a
$G$-module. By Example~\ref{exa: relatively injective modules} it is
relatively injective. Further, the contractibility of $X$ implies
that $\hom_\bbC(C_\ast^{\CW}(X),E)$ is a
resolution of $E$. The next statement follows
from Theorem~\ref{thm: fundamental lemma}.

\begin{theorem}\label{thm: geometric model and cohomology}
	Let $X$ be a geometric model for a totally disconnected group~$G$.
	Then
	\[
		H^n(G, E)\cong H^n(\hom_\bbC(C_\ast^{\CW}(X), E)^G).
	\]
\end{theorem}

Not every discrete or totally disconnected group has a geometric model. For discrete groups
having a geometric model means that the group satisfies the finiteness condition $F_\infty$.
But by attaching enough equivariant cells
to increase connectivity one can show~\cite[1.2]{lueck-classifying}:

\begin{theorem}
	For every totally disconnected $G$ there is a contractible $G$-CW-complex
	whose stabilisers are open and compact.
\end{theorem}

More is true but not needed here:
Every totally disconnected group has a classifying space for
the family of compact-open subgroups.

What about the opposite case of a connected Lie group $G$?
There we lack geometric models to compute continuous cohomology, but we can use
the infinitesimal structure. The van-Est isomorphism relates continuous cohomology and
Lie algebra cohomology. We won't discuss it here and instead refer to~\cite{guichardet}.

\subsection{Definition and geometric interpretation of $\ell^2$-cohomology}\label{sub: l2 cohomology}
Let $G$ be a second countable, locally compact group. For convenience, we require that the
the Haar measure~$\mu$ is unimodular, i.e.~invariant under left and right translations. Unimodularity
becomes a necessary assumption in Section~\ref{sec: dimension}.

We consider continuous cohomology in the $G$-module $L^2(G)$ which consists
of measurable square-integrable functions on $G$ modulo null sets.
A left $G$-action on $L^2(G)$ is given by $(gf)(h)=f(g^{-1}h)$.

The $G$-module $L^2(G)$
also carries the right $G$-action $(fg)(h)=f(hg^{-1})$, which commutes with the left action.
In Section~\ref{sec: cohomology} we ignore this right action. It becomes important
in Sections~\ref{sec: dimension} and~\ref{sec: betti numbers} when we define $\ell^2$-Betti
numbers.

\begin{defn}
	We call $H^\ast(G,L^2(G))$ and $\bar H^\ast(G, L^2(G))$
	the \emph{$\ell^2$-cohomology} and the \emph{reduced $\ell^2$-cohomology} of $G$, respectively.
\end{defn}

Assume that $G$ possesses a geometric model~$X$.
By Theorem~\ref{thm: geometric model and cohomology} the $\ell^2$-cohomology of $G$
can be expressed as the cohomology of the $G$-invariants of the chain complex
$\hom_\bbC(C_\ast^{\CW}(X), L^2(G))$. In view of~\eqref{eq: geometric chain group},
this is a chain complexes of Hilbert spaces
\begin{equation}\label{eq: concrete chain group of geometric model}
	\hom_\bbC(C_n^{\CW}(X), L^2(G))^G\cong \bigoplus_{U\in\calF_n}C(G/U, L^2(G))^G\cong \bigoplus_{U\in\calF_n}L^2(G)^U
\end{equation}
with bounded differentials.

Let us rewrite this chain complex in a way so that the
group $G$ does not occur anymore:
As a (non-equivariant) cellular complex the $n$-th cellular chain group
$C_n^{\CW}(X)$ comes with a preferred basis $B_n$,
given by $n$-cells, which is unique
up to permutation and signs. We define the vector space
\[
	\ell^2C^n_{\CW}(X):=\bigl\{f: C_n^{\CW}(X)\to\bbC\mid \sum_{e\in B_n}|f(e)|^2<\infty\bigr\}\subset \hom_\bbC(C_n^{\CW}(X), \bbC)
\]
of $\ell^2$-cochains in the cellular cochains; it has the structure of a Hilbert
space with Hilbert basis $\{f_e\mid e\in B_n\}$ where $f_e(e')=1$ for $e'=e$ and
$f_e(e')=0$ for $e'\in B_n\bs\{e\}$. For general cellular complexes
the differentials
in the cellular cochain complex
are not bounded as operators, but
in the presence of a cocompact group action on skeleta they are.

\begin{defn}
	The \emph{(reduced) $\ell^2$-cohomology} of $X$ is defined as the (reduced) cohomology of
	$\ell^2C_{\CW}^\ast(X)$ and denoted by $\ell^2 H^\ast(X)$ or
	$\ell^2\bar H^\ast(X)$, respectively.
\end{defn}

\begin{prop}\label{prop: geometric description of l2 cohomology}
	The $\ell^2$-cohomology of $G$ and the $\ell^2$-cohomology of a geometric
	model of $G$ are isomorphic. Similarly for the reduced cohomology.
\end{prop}

\begin{proof}
	The maps
	\[
		\hom_\bbC(C_n^{\CW}(X), L^2(G))^G\to \ell^2C^n_{\CW}(X).
	\]
	that take $f$ to the $\ell^2$-cochain that assigns to $e\in B_n$ the
	essential value of the essentially constant function $f(e)\vert_{U}$, where
	$U<G$ is the open stabiliser of $e$,
	form a chain isomorphism. The claim follows now from
	Theorem~\ref{thm: geometric model and cohomology}.
\end{proof}

\subsection{Reduced $\ell^2$-cohomology and harmonic cochains}\label{subsub: harmonic}

\begin{definition}
	Let $W^0\xrightarrow{d^0}W^1\xrightarrow{d^1}W^2\ldots$
	be a cochain complex of Hilbert spaces such that the differentials are
	bounded operators. The \emph{Laplace-operator} in degree~$n$ is
	the bounded operator
	$\Delta^n=(d^n)^\ast\circ d^n+d^{n-1}\circ (d^{n-1})^\ast: W^n\to W^n$.
	Here $\ast$ means the adjoint operator.
	A cochain $c\in W^n$ is \emph{harmonic} if $\Delta^n(c)=0$.
\end{definition}

\begin{prop}\label{prop: harmonic cochains}
	Every harmonic cochain is a cocycle, and inclusion induces
	a topological isomorphism $\ker(\Delta^n)\xrightarrow{\cong} \bar H^n(W^\ast)=\ker(d^n)/\closure(\im(d^{n-1}))$.
\end{prop}

\begin{proof}
	Since $\Delta^n$ is a positive operator, we have $c\in\ker(\Delta^n)$ if and only
	if \[\langle \Delta^n(c), c\rangle=\norm{d^nc}^2+\norm{(d^{n-1})^\ast(c)}^2=0.\]
	Hence $\ker(\Delta^n)=\ker(d^n)\cap\ker((d^{n-1})^\ast)$. The second statement
	follows from
	\begin{align*}
		W^p &= (\ker(d^n)\cap \im(d^{n-1})^\perp)\oplus \closure(\im(d^{n-1}))\oplus \ker(d^n)^\perp\\
			&= (\ker(d^n)\cap \ker((d^{n-1})^\ast))\oplus  \closure(\im(d^{n-1}))\oplus \ker(d^n)^\perp\\
			&= \ker(\Delta^n)\oplus  \closure(\im(d^{n-1}))\oplus \ker(d^n)^\perp.\qedhere
	\end{align*}

\end{proof}

As a consequence of Propositions~\ref{prop: geometric description of l2 cohomology}
and~\ref{prop: harmonic cochains}, we obtain:

\begin{corollary}
	Let $X$ be a geometric model of~$G$. The space of harmonic $n$-cochains of
	$\ell^2 C^\ast_{\CW}(X)$ is isomorphic to $\bar H^n(G, L^2(G))$.
\end{corollary}

\begin{example}\label{exa: free group}
The $4$-regular tree is a geometric
model of the free group~$F$ of rank two.
Let $d: \ell^2C^0_{\CW}(X)\to \ell^2C^1_{\CW}(X)$ be the
differential. Since there are no $2$-cells, the first Laplace operator
$\Delta^1$ is just $dd^\ast$. A $1$-cochain is harmonic if and only if it is
in the kernel of $d^\ast$.
We choose a basis $B$ of $C_1^{\CW}(X)$ by
orienting $1$-cells in the following way:
				\begin{center}\begin{tikzpicture}[scale=0.7]
					\draw[white] (0,5) -- (-5,0) -- (0,-5) -- (5,0) -- cycle;
					\draw (0,5) -- (0,-5);
					\draw (-5,0) -- (5,0);
					\draw (-2,3) -- (2,3);
					\draw (-3,-2) -- (-3,2);
					\draw (-2,-3) -- (2,-3);
					\draw (3,-2) -- (3,2);
					\draw (-3.5,1.5) -- (-2.5,1.5);
					\draw (-4.5,0.5) -- (-4.5,-0.5);
					\draw (-3.5,-1.5) -- (-2.5,-1.5);
					\draw (-0.5,-4.5) -- (0.5,-4.5);
					\draw (-1.5,-2.5) -- (-1.5,-3.5);
					\draw (1.5,-2.5) -- (1.5,-3.5);
					\draw (2.5,-1.5) -- (3.5,-1.5);
					\draw (4.5,0.5) -- (4.5,-0.5);
					\draw (2.5,1.5) -- (3.5,1.5);
					\draw (-1.5,3.5) -- (-1.5,2.5);
					\draw (-0.5,4.5) -- (0.5,4.5);
					\draw (1.5,3.5) -- (1.5,2.5);

					\draw[->] (-3,0.7) -- (-3,0.8) node[left]{$0$};
					\draw[->] (-3,-0.8) -- (-3,-0.7) node[right]{$1/4$};

					\draw[->] (-3.9, 0) -- (-3.8,0) node[below]{$1/4$};
					\draw[->] (1.4,0) -- (1.5,0) node[below]{$1/4$};
					\draw[->] (-1.6,0) -- (-1.5,0) node[below]{$1/2$};
					\draw[->] (0,1.4) -- (0,1.5) node[right]{$1/4$};
					\draw[->] (0,-1.6) -- (0,-1.5) node[right]{$0$};
					\draw[->] (0,3.7) -- (0,3.8) node[right]{$1/8$};
					\draw[->] (0.7,3) -- (0.8,3) node[below]{$1/8$};
					\draw[->] (-0.9,3) -- (-0.8,3) node[below]{$0$};
					\draw[->] (3,0.7) -- (3,0.8) node[left]{$1/8$};
					\draw[->] (3,-0.9) -- (3,-0.8) node[left]{$0$};
					\draw[->] (3.7,0) -- (3.8,0) node[below]{$1/8$};
					\filldraw (-3,1.5) circle (2pt);
					\filldraw (0,-4.5) circle (2pt);
					\filldraw (1.5,-3) circle (2pt);
					\filldraw (-1.5,-3) circle (2pt);
					\filldraw (0,-3) circle (2pt);
					\filldraw (1.5,3) circle (2pt);
					\filldraw (-1.5,3) circle (2pt);
					\filldraw (0,4.5) circle (2pt);
					\filldraw (0,3) circle (2pt);
					\filldraw (-3,-1.5) circle (2pt);
					\filldraw (-4.5,0) circle (2pt);
					\filldraw (-3,0) circle (2pt);
					\filldraw (3,-1.5) circle (2pt);
					\filldraw (3,1.5) circle (2pt);
					\filldraw (0,0) circle (2pt);
					\filldraw (4.5,0) circle (2pt);
					\filldraw (3,0) circle (2pt);
				\end{tikzpicture}\end{center}
For $e\in B$ let $e(-)$ be the
starting point and $e(+)$ the end point of~$e$. Then:
\begin{align*}
	d(f)(e)&=f(e(+))-f(e(-))\\
	d^\ast(g)(v)&=\sum_{e(+)=v}g(e)-\sum_{e(-)=v}g(e)
\end{align*}
In the second equation the sums run over all edges whose starting or end point
is~$v$. The picture indicates a non-vanishing $c\in \ell^2C^1_{\CW}(X)$
with $d^\ast(c)=0$, thus $\Delta^1(c)=0$. Therefore $\bar H^1(F, \ell^2(F))\ne 0$.
\end{example}

\section{von-Neumann algebras, trace and dimension}\label{sec: dimension}

Throughout, we fix a second countable, locally compact group~$G$ with
left invariant Haar measure~$\mu$. We require that $G$ is unimodular, that is, $\mu$
is also right invariant.

\subsection{The von Neumann algebra of a locally compact group} 
\label{sub:the_von_neumann_algebra_of_a_locally_compact_group}
The continuous functions with compact support $C_c(G)$ form a $\bbC$-algebra with involution through
the convolution product
\[
	(f\ast g)(s)=\int_G f(t)g(t^{-1}s)d\mu(t)
\]
and the involution $f^\ast(s)=\overline{f(s^{-1})}$.
Taking convolution with $f\in C_c(G)$ is still defined for a function $\phi\in L^2(G)$.
It follows from the integral Minkowski inequality that
\[
	\norm{f\ast \phi}_2\le \norm{f}_1\norm{\phi}_2,
\]
where $\norm{\_}_p$ denotes the $L^p$-Norm.
We obtain a $\ast$-homomorphism into the algebra of bounded operators on $L^2(G)$:
\[
	\lambda: C_c(G)\to \calB(L^2(G)),~\lambda(f)(\phi)= f\ast \phi
\]
Similarly, we obtain a $\ast$-anti-homomorphism
\[
	\rho: C_c(G)\to  \calB(L^2(G)),~\rho(f)(\phi)=\phi\ast f.
\]

\begin{defn}
	A \emph{von-Neumann algebra} is a subalgebra of the bounded operators of
	a Hilbert space that is closed in the weak operator topology and closed under taking adjoints. The \emph{von-Neumann algebra $L(G)$ of the group $G$} is the von-Neumann algebra defined as the weak closure
	of $\im(\lambda)$.
\end{defn}

The weak closure $R(G)$ of $\im(\rho)$ in $\calB(L^2(G))$ is the commutant of $L(G)$ inside
$\calB(L^2(G))$.

\begin{remark}
	Consider the operator $u_g\in \calB(L^2(G))$ such that
	\[
		u_g(\phi)(s)=\phi(g^{-1}s).
	\]
	Similarly, we define $r_g(\phi)(s)=\phi(sg)$.
	We claim that $u_g\in L(G)$.
	If $G$ is discrete, then the Kronecker function $\delta_g$ on $G$ is
	continuous, so $\delta_g\in C_c(G)$, and we have $u_g:=\lambda(\delta_g)\in L(G)$. If $G$ is not discrete,
	we can choose a sequence $f_n\in C_c(G)$ of positive functions with $\norm{f_n}_1=1$ whose supports tend to $g$.
	Then $(\lambda(f_n))_{n\in\bbN}$ converges strongly, thus weakly, to $u_g$, implying $u_g\in L(G)$.
	One can also show that $L(G)$ is the weak closure of the span of $\{u_g\mid g\in G\}$.
	Similary it follows that $r_g\in R(G)$
\end{remark}

\begin{remark}\label{rem: anti-isomorphism}
Let $j: L^2(G)\to L^2(G)$ be the conjugate linear isometry
with \[j(\phi)(s)=\phi^\ast(s)=\overline{\phi(s^{-1})}.\]
Then
\[
	J: L(G)\to R(G),~ J(T)=j\circ T^\ast\circ j
\]
is a $\ast$-anti-isomorphism, that is, $J(T^\ast)=J(T)^\ast$ and $J(T\circ S)=J(S)\circ J(T)$.
Furthermore, we have $J(\lambda_g)=r_g$ and
$J(\lambda(f))=J(\rho(f))$ for $f\in C_c(G)$.
\end{remark}
\begin{defn}[$L(G)$-module structures]
	The Hilbert space $L^2(G)$ is naturally a left $L(G)$-module via $T\cdot \phi=T(\phi)$ for
	$T\in L(G)$ and $\phi\in L^2(G)$.
	The left $L(G)$-module structure restricts to the left $G$-module structure via $g\cdot \phi:=u_g(\phi)$.
    The Hilbert space $L^2(G)$ becomes a right $L(G)$-module by the anti-isomorphism~$J$, explicitly
	$\phi\cdot T:=J(T)(\phi)$.
	In the sequel, $L^2(G)$ will be regarded as a bimodule endowed
	with the left $G$-module structure and the right $L(G)$-module structure.
\end{defn}

\subsection{Trace}\label{sub: trace}
We explain the semifinite trace on the von Neumann algebra of a locally compact group.
This a bit technical, but indispensable for the dimension theory.
We start by discussing traces on an arbitrary von-Neumann algebra~$\calA$.

Let $\calA_+$
be the subset of positive operators in $\calA$. For $S,T\in\calA_+$
one defines $S\le T$ by $T-S\in\calA_+$; it is a partial order on $\calA_+$.
Every bounded totally ordered subset of $\calA_+$ has a supremum in $\calA_+$.

\begin{definition}\label{def: trace}
	A \emph{trace} on a von Neumann algebra $\calA$ is a function $\tau: \calA_+\to [0,\infty]$ such that
	\begin{enumerate}
		\item $\tau(S)+\tau(T)=\tau(S+T)$ for $S,T\in\calA_+$;
		\item $\tau(\lambda S)=\lambda\tau(S)$ for $S\in\calA_+$ and $\lambda\ge 0$;
		\item $\tau(SS^\ast)=\tau(S^\ast S)$ for $S\in\calA$.
	\end{enumerate}
	Let $\calA_+^\tau=\{S\in\calA_+\mid \tau(S)<\infty\}$.
    A trace $\tau$ is \emph{faithful} if $\tau(T)>0$ for every $T\in\calA_+\bs\{0\}$.
	It is \emph{finite} if $\calA_+^\tau=\calA_+$.
	It is \emph{semifinite} if $\calA_+^\tau$ is weakly dense in $\calA_+$.
	It is \emph{normal} if the supremum of traces of a bounded totally ordered subset in $\calA_+$
	is the trace of the supremum.
\end{definition}

If $\calA\subset\calB(H)$ is a von-Neumann algebra with trace~$\tau$, then
the $n\times n$-matrices $M_n(\calA)\subset\calB(H^n)$ are a von-Neumann algebra
with trace \[(\tau\otimes\id_n)(S):=\tau(S_{11})\ldots+\tau(S_{nn}).\]

\begin{remark}\label{rem: finite traces}
	Let $\calA^\tau$ be the linear span of $\calA_+^\tau$. It is clear that $\tau$
	extends linearly to $\calA^\tau$. If $\tau$ is finite, we have $\calA^\tau=\calA$, so $\tau$
	is defined on all of $\calA$. Further, $\calA^\tau$ is always an ideal in $\calA$~\cite[p.~318]{takesaki}.
	The \emph{trace property}
	\[
		\tau(ST)=\tau(TS)\text{ for all $S\in \calA^\tau$ and $T\in\calA$.}
	\]
	holds true. Its deduction from the third property in
	Definition~\ref{def: trace} takes a few lines and uses polarisation identities.
	See~\cite[Lemma~2.16 on p.~318]{takesaki}.
\end{remark}

After this general discussion we turn again to the von-Neumann algebra of~$G$. We call $\phi\in L^2(G)$ \emph{left bounded} if there
is a bounded operator, denoted by $\lambda_\phi$ on $L^2(G)$ such that
$\lambda_\phi(f)=\phi\ast f$ for every $f\in C_c(G)$. Of course, every element
$f\in C_c(G)$ is left bounded and $\lambda_f=\lambda(f)$.
Define for an element $S^\ast S\in L(G)_+$ (every positive
operator can be written like this):
\begin{equation}\label{eq: trace on locally compact group}
	\tau_{(G,\mu)}(S^\ast S)=\begin{cases}
		\norm{\phi}_2^2 & \text{ if there is a left bounded $\phi\in L^2(G)$ with $\lambda_\phi=S$;}\\
		\infty & \text{otherwise.}
	\end{cases}
\end{equation}

\begin{notation}
	The Haar measure is only unique up to scaling. Hence we keep $\mu$ in the
	notation~$\tau_{(G,\mu)}$. If $G=\Gamma$ is discrete, we always take the counting measure as
	Haar measure and simply write $\tau_\Gamma$.
\end{notation}

See~\cite[7.2.7~Theorem]{pedersen-book} for a proof of the following fact.

\begin{theorem}
	$\tau_{(G,\mu)}$ is a faithful normal semifinite trace on $L(G)$.
\end{theorem}

Let us try to obtain a better understanding of
Defintion~\eqref{eq: trace on locally compact group}.
To this end, we first consider the case that $G=\Gamma$ is a discrete group. Then
$\delta_e\in C_c(\Gamma)=\bbC[\Gamma]\subset L^2(\Gamma)$. For every $S\in L(G)$ it is
$S(\delta_e)\in L^2(\Gamma)$ and $S=\lambda_{S(\delta_e)}$. This implies that
every element in $L(\Gamma)_+$ has finite $\tau_\Gamma$-trace. Hence $\tau_\Gamma$ is finite.
Further, from
\[
	\tau_\Gamma(S^\ast S)=\norm{S(\delta_e)}_2^2=\langle S(\delta_e), S(\delta_e)\rangle_{L^2(\Gamma)}=\langle S^\ast S(\delta_e),\delta_e\rangle_{L^2(\Gamma)}
\]
we conclude and record:

\begin{remark}[Trace for discrete groups]
	If $G=\Gamma$ is discrete, then $\tau_\Gamma$ is finite and thus everywhere defined. For
	every $T\in L(\Gamma)$ we have $\tau_\Gamma(T)=\langle T(\delta_e),\delta_e\rangle_{L^2(\Gamma)}$.
\end{remark}

\begin{remark}[Trace for totally disconnected groups]
Let $G$ be totally disconnected. Then we have a decreasing neighborhood basis
$(K_n)$ by open-compact subgroups. Then $\frac{1}{\mu(K_n)}\lambda(\chi_{K_n})\in L(G)$
is a projection (see Example~\ref{exa: compact-open subgroup}).
Let $S\in L(G)$.
Then $\phi_n:=S(\chi_{K_n})\in L^2(G)$ is a left bounded
element such that $\lambda_{\phi_n}=S\circ\lambda(\chi_{K_n})$. From that
and~\eqref{eq: trace on locally compact group} it is easy to see that
\begin{equation*}
\tau_{(G,\mu)}(S^*S) = \lim_{n\to\infty} \frac{1}{\mu(K_n)^2} \cdot \lVert S(\chi_{K_n}) \rVert^2_2\in [0,\infty]
\end{equation*}
\end{remark}

Now back to general~$G$:

\begin{remark}[Trace for arbitrary groups]
Because of $\lambda_f=\lambda(f)$ for every $f\in C_c(G)$ we obtain
that $\tau_{(G,\mu)}(\lambda(f)^\ast\lambda(f))=\norm{f}^2_2$. One quickly verifies that the latter is
just evaluation at the unit element: $\norm{f}^2_2=(f^\ast f)(e)$. It turns out that
$C_c(G)\subset L(G)^{\tau_{(G,\mu)}}$, and the trace $\tau_{(G,\mu)}$ is evaluation at $e\in G$ on $C_c(G)$.
\end{remark}
\subsection{Dimension}

We explain first L\"uck's dimension function
over a von-Neumann algebra endowed with a finite trace.
We refer
to~\cite{lueck-book,lueck-dimension} for proofs. L\"uck's work was a
major advance in creating a general and algebraic theory of $\ell^2$-Betti numbers.
An alternative algebraic approach was developed
by Farber~\cite{farber}.
After that
we discuss Petersen's generalisation to von-Neumann algebras with semifinite
traces~\cite{petersen}.

\subsubsection{Finite traces}
Let $\tau$ be a finite normal faithful trace on a von-Neumann algebra~$\calA$.
According to Remark~\ref{rem: finite traces}, $\tau$ is a functional on $\calA$
which satisfies the trace property $\tau(ST)=\tau(TS)$ for $S,T\in\calA$.
We start by explaining the dimension for finitely generated projective right $\calA$-modules.
Let $P$ be such a module. This means
that $P$ is isomorphic to the image of left multiplication $l_M: \calA^n\to \calA^n$
with an idempotent matrix $M\in M_n(\calA)$. In general the sum $\sum_{i=1}^n M_{ii}$
of the diagonal entries of $M$ depends not only on~$P$ but on the specific choice of~$M$.
However, Hattori and Stallings observed that
its image in the quotient $\calA/[\calA, \calA]$ by the additive subgroup
generated by all commutators is independent of the choices~\cite[Chapter~IX]{brown}.
This is true for
an arbitrary ring~$\calA$. The trace $\tau$ is still defined on the
quotient (see Remark~\ref{rem: finite traces}).
Therefore one defines:

\begin{definition}
	The \emph{Hattori-Stallings rank} $\hs(P)\in \calA/[\calA,\calA]$
	of a finitely generated projective right $\calA$-module
	$P$ is defined as the image of $\sum_{i=1}^n M_{ii}$ in $\calA/[\calA,\calA]$
	for any idempotent matrix $M\in M_n(\calA)$ with $P\cong \im(l_M)$.
	The \emph{dimension} of~$P$ is defined as
	\[\dim_\tau(P)=\tau(\hs(P))=(\tau\otimes\id_n)(M)\in [0,\infty).\]
\end{definition}

Henceforth $\calA$-modules are understood to be right $\calA$-modules.

\begin{remark}
	Let $\ell^2(\calA,\tau)$ be the GNS-construction of $\calA$ with respect to~$\tau$,
	that is the completion of the pre-Hilbert space $\calA$ with inner product
	$\langle S,T\rangle_\tau=\tau(ST^\ast)$.
	By an observation of Kaplansky~\cite[Lemma~6.23 on p.~248]{lueck-book},
	every finitely generated projective
	$\calA$-module can be described by a projection matrix $M$, that is a
	matrix $M$ with $M^2=M$ \emph{and} $M^\ast=M$. Then $M$ yields a right
	Hilbert $\calA$-submodule of $\ell^2(\calA,\tau)^n$, namely the image of
	left multiplication $\ell^2(\calA,\tau)^n\to \ell^2(\calA,\tau)^n$ with $A$.
	And $\dim_\tau(M)$
	coincides with the von-Neumann dimension of this Hilbert $\calA$-submodule.
	See~\cite[Chapter~1]{lueck-book} for more information on Hilbert $\calA$-modules.
	The classic reference is~\cite{neumann}.
\end{remark}

The idea of how to extend $\dim_\tau$ to arbitrary modules is almost naive; the difficulty lies in
showing its properties.

\begin{definition}
	Let $M$ be an arbitrary $\calA$-module. Its \emph{dimension} is
	defined as
	\[
		\dim_\tau(M)=\sup\{\dim_\tau(P)\mid \text{$P\subset M$ fin.~gen.~projective submodule}\}\in [0,\infty].
	\]
\end{definition}

First of all, using the same notation $\dim_\tau$ as before requires a justification.
But indeed, the new definition coincides with the old one on finitely generated
projective modules. And so $\dim_\tau(\calA)=\tau(1_\calA)$ which we usually normalise
to be~$1$. The following two properties are important and, in the end, implied by
the additivity and normality of~$\tau$.

\begin{theorem}[Additivity]
	If $0\to M_1\to M_2\to M_3\to 0$ is a short exact sequence of $\calA$-modules, then
	$\dim_\tau(M_2)=\dim_\tau(M_1)+\dim_\tau(M_3)$. Here $\infty+a=a+\infty=\infty$ for $a\in [0,\infty]$ is understood.
\end{theorem}

\begin{theorem}[Normality]
	Let $M$ be an $\calA$-module, and let $M$ be the union of an increasing sequence of $\calA$-submodules~$M_i$.
	Then $\dim_\tau (M)=\sup_{i\in\bbN}\dim_\tau(M_i)$
\end{theorem}

\begin{example}
The only drawback of $\dim_\tau$ -- in comparison to the dimension of vector spaces -- is that
$\dim_\tau(N)=0$ does not, in general, imply $N=0$.
Here is an example. Take a standard probability space $(X,\nu)$.
Then the $\nu$-integral is a finite trace $\tau$ on $\calA=L^\infty(X)$.
Let $X=\bigcup_{i=1}^n X_i$ be an increasing union of measurable sets such that $\nu(X_i)<1$
for every $i\in\bbN$. Each characteristic function $\chi_{X_i}$ is a projection in $\calA$
with trace $\nu(X_i)$. And each $\chi_{X_i}\calA=L^\infty(X_i)$ is a finitely generated projective module of dimension $\tau(\chi_{X_i})=\nu(X_i)$. Let $M\subset\calA$ be the increasing union of
submodules $\chi_{X_i}\calA$. Then $\calA/M\ne 0$ but $\dim_\tau (\calA/M)=\dim_\tau(\calA)-\dim_\tau(M)=0$ by additivity and normality.
\end{example}

\subsubsection{Semifinite traces}

Let $\calA$ be a von-Neumann algebra with a semifinite trace~$\tau$.
The definition of the dimension function only needs small modifications,
and the proof of being well-defined
and other properties run basically like the one in the finite
case~\cite[Appendix~B]{petersen-phd}.

We say that a projective $\calA$-module $P$ is
\emph{$\tau$-finite} if $P$ is finitely generated and if one, then any, representing projection matrix $M\in M_n(\calA)$,
$P\cong \im(l_M: \calA^n\to\calA^n)$, satisfies $(\tau\otimes\id_n)(M)<\infty$. The
\emph{dimension} $\dim_\tau(P)\in [0,\infty)$ of a $\tau$-finite projective $\calA$-module
$P\cong \im(l_M)$ is defined as $(\tau\otimes\id_n)(M)$.

\begin{definition}
		Let $M$ be an arbitrary $\calA$-module. Its \emph{dimension} is
		defined as
		\[
			\dim_\tau(M)=\sup\{\dim_\tau(P)\mid \text{$P\subset M$ $\tau$-finite projective submodule}\}\in [0,\infty].
		\]
\end{definition}

Similarly as before, the notation is consistent with the one for $\tau$-finite projective modules,
and additivity and normality hold true.

\begin{remark}
 While the situation in the semifinite case is quite similar to the finite case, there is one important difference: If $\tau$ is not finite, then $\dim_\tau(\calA)=\tau(1_\calA)=\infty$.
\end{remark}

\begin{notation}
	We write $\dim_{(G,\mu)}$ instead of $\dim_{\tau_{(G,\mu)}}$ and $\dim_\Gamma$ instead of
	$\dim_{\tau_\Gamma}$.
\end{notation}

\begin{example}\label{exa: compact-open subgroup}
	Let $K<G$ be an open-compact subgroup. Then the characteristic function $\chi_K$ is continuous and
	$P=\lambda(\frac{1}{\mu(K)}\chi_K)$ is a projection in $L(G)$. According to
	the remark at the end of Subsection~\ref{sub: trace}, we have $\tau_{(G,\mu)}(P)=1/\mu(K)$, thus
	the dimension of the projective $L(G)$-module $P\cdot L(G)$ is $1/\mu(K)$. The $L(G)$-module
	$P\cdot L^2(G)$ is not projective, but, by an argument involving rank completion,
	one can show~\cite[B.25~Proposition]{petersen-phd} that
	\begin{equation}\label{eq: rank completion dim}
		\dim_{(G,\mu)}(P\cdot L^2(G))=\dim_{(G,\mu)}(P\cdot L(G))=\frac{1}{\mu(K)}.
	\end{equation}
	One easily verifies that $P\cdot L^2(G)$ consists of all left $K$-invariant functions, so
	\[P\cdot L^2(G)=L^2(G)^K.\]
\end{example}

\section{$\ell^2$-Betti numbers of groups}\label{sec: betti numbers}

Throughout, $G$ denotes a second countable, locally compact, unimodular group.

\subsection{Definition}
In the definition of $\ell^2$-cohomology we only used the left $G$-module structure.
The right $L(G)$-module structure survives the process of taking $G$-invariants of the bar resolution, and so
$H^\ast(G,L^2(G))$ inherits a right $L(G)$-module structure from the one of $L^2(G)$. The following definition is due to Petersen for second countable, unimodular, locally compact~$G$. It is modelled after and coincides with L\"uck's definition for discrete~$G$.

\begin{definition}
	The $\ell^2$-Betti numbers of~$G$ with Haar measure~$\mu$
	are defined as
	\[
		\betti_p(G,\mu):=\dim_{(G,\mu)} \bigl(H^p(G,L^2(G))\bigr)\in [0,\infty].
	\]
\end{definition}

\begin{remark}
	In L\"uck's book~\cite{lueck-book}, where only the case of discrete~$G$ is discussed,
	the $\ell^2$-Betti numbers are defined as $\dim_G\bigl(H_p(G,L(G))\bigr)$.
	By~\cite[Theorem~2.2]{peterson+thom} L\"uck's definition coincides with the one above.
\end{remark}

The homological algebra in Section~\ref{sec: cohomology} can be carried
through such that the additional right $L(G)$-module structure is respected.
In particular, we obtain (see Theorem~\ref{thm: geometric model and cohomology}):

\begin{theorem}\label{thm: betti and geometric model}
	Let $G$ be totally disconnected and $X$ a geometric model of $G$.
	Then
	\[
		\betti_p(G,\mu)=\dim_{(G,\mu)}\bigl(H^p\bigl(\hom_\bbC(C_\ast^{\CW}(X), L^2(G))^G\bigr)\bigr).
	\]
\end{theorem}

For totally disconnected groups we can compute the $\ell^2$-Betti numbers also through
the reduced cohomology~\cite[]{petersen-phd}:

\begin{theorem}\label{thm: betti and reduced}
	Let $G$ be totally disconnected. Then
	\[
		\betti_p(G,\mu)=\dim_{(G,\mu)} \bigl(\bar H^p(G,L^2(G))\bigr)
	\]
	If $X$ is a geometric model of~$G$, then
	\[
		\betti_p(G,\mu)=\dim_{(G,\mu)}\bigl(\bar H^p\bigl(\hom_\bbC(C_\ast^{\CW}(X), L^2(G))^G\bigr)\bigr).
	\]
\end{theorem}

Many properties of $\ell^2$-Betti numbers of discrete groups possess analogues for locally compact groups. We refer to~\cite{petersen-phd} for more information and discuss only the Euler-Poincare formula:

Let $X$ be a cocompact proper $G$-CW complex. Let $K_1,\ldots, K_n<G$
be the stabilisers of $G$-orbits of $p$-cells of~$X$.
The \emph{weighted number
of equivariant $p$-cells} of $X$ is then defined as
\[ c_p(X;G,\mu)=\mu(K_1)^{-1}+\cdots +\mu(K_n)^{-1}.\]

\begin{definition}
The \emph{equivariant Euler characteristic} of $X$ is defined as
\[ \chi(X;G,\mu):=\sum_{p\ge 0}(-1)^p c_p(X;G).\]
\end{definition}

\begin{theorem}[Euler-Poincare formula]\label{thm: euler poincare}
Let $G$ be totally disconnected, and let $X$ be a cocompact geometric
model of~$G$. Then
\[ \sum_{p\ge 0}(-1)^p\betti_p(G,\mu)=\chi(X;G,\mu).\]
\end{theorem}

\begin{proof}
Let $C^\ast:= \hom_\bbC(C_\ast^{\CW}(X), L^2(G))^G$, and let $H^\ast$ be the cohomology of $C^\ast$.
Let $Z^p$ be the cocycles in $C^p$ and
$B^p$ be the coboundaries in $C^p$.
Note that $c_p(X;G,\mu)=\dim_{(G,\mu)}(C^p)$. We have exact sequences
$0\to Z^p\to C^p\to B^{p+1}\to 0$ and $0\to B^p\to Z^p\to H^p\to 0$
of $L(G)$-modules. By additivity of $\dim_{(G,\mu)}$ we conclude that
	\begin{align*}
		\chi(X;G,\mu) &= \sum_p (-1)^p\dim_{(G,\mu)}(C^p)\\
		 &= \sum_p (-1)^p(\dim_{(G,\mu)}(Z^p)+\dim_{(G,\mu)} (B^{p+1}))\\
		                                                 &= \sum_p (-1)^p(\dim_{(G,\mu)}(B^p)+\dim_{(G,\mu)}(H^p)+\dim_{(G,\mu)} (B^{p+1}))\\
		                                                 &= \sum_p (-1)^p\betti_p(G,\mu).\qedhere
		                                             \end{align*}
\end{proof}

\begin{remark}\label{rem: vanishing of betti}
	It turns out that
	\[\betti_p(G,\mu)>0 \Leftrightarrow \bar H^p\bigl(\hom_\bbC(C_\ast^{\CW}(X), L^2(G))^G\bigr)\ne 0.\]
	In general, this is false for non-reduced continuous cohomology. So constructing non-vanishing harmonic cocycles is a
	way to show non-vanishing of $\ell^2$-Betti numbers (see Subsection~\ref{subsub: harmonic}).
	Example~\ref{exa: free group} shows that the first $\ell^2$-Betti number of a non-abelian free
	group is strictly positive.
\end{remark}

\subsection{Quasi-isometric and coarse invariance}\label{sub: qi}
\label{sub:the_coarse_geometric_aspect}

Every locally compact, second countable group $G$ possesses a left-invariant proper continuous metric~\cite{struble}, and any two left-invariant proper continuous metrics on $G$ are coarsely equivalent. Thus  
$G$ comes with a well defined coarse geometry. If $G$ is compactly generated, then 
the word metric associated to a symmetric compact generating set is coarsely 
equivalent to any left-invariant proper continuous metric. Finally, any coarse equivalence between compactly generated second countable groups is a quasi-isometry with respect to word metrics of compact symmetric generating sets. A recommended background reference for these notions is~\cite{cornulier-book}*{Chapter~4}.  

As discussed in the introduction, neither the exact values nor the proportionality of $\ell^2$-Betti numbers are coarse invariants. But the vanishing of $\ell^2$-Betti numbers is a coarse invariant. The following result, which was proved by Pansu~\cite{pansu} for discrete groups admitting a classifying space of finite type, by Mimura-Ozawa-Sako-Suzuki~\cite{taka}*{Corollary~6.3} for all countable discrete groups, and by Schr\"odl and the author in full generality~\cite{sauer+schroedl}, has a long history; 
we refer to the introduction in~\cite{sauer+schroedl} for an historical overview. 

\begin{theorem}
The vanishing of the $n$-th $\ell^2$-Betti number of a unimodular, second countable, locally compact group is an invariant of coarse equivalence. 
\end{theorem}

To give an idea why this is true, let us consider compactly generated, second countable, totally disconnected, locally compact, unimodular groups $G$ and $H$.  
We assume, in addition, that $G$ and $H$ admit cocompact (simplicial) geometric models $X$ and $Y$, respectively. Endowed with the simplicial path metrics induced by the Euclidean metric on standard simplices, $X$ and $Y$ are quasi-isometric to $G$ and $H$, respectively. 

By Proposition~\ref{prop: geometric description of l2 cohomology} and
Remark~\ref{rem: vanishing of betti} we obtain that
\[
	\betti_p(G)>0\Leftrightarrow \ell^2 \bar H^p(X)\ne 0.
\]
Similarly for $H$ and $Y$. 
Since $X$ and $Y$ are quasi-isometric and both are uniformly contractible, 
the connect-the-dots technique (see e.g.~\cite[Proposition~A.1]{brady+farb})
yields a Lipschitz homotopy equivalence $f\colon X\to Y$.
Pansu~\cite{pansu} proves that from that we obtain an isomorphism
\[ \ell^2 \bar H^p(X)\cong \ell^2 \bar H^p(Y)\]
in all degrees~$p$, implying the above theorem. 

The phenomenon that group homological invariants can be viewed as
coarse-geometric invariants is not unique to the theory of $\ell^2$-cohomology or $\ell^2$-Betti numbers, of course.
For instance, $H^\ast(\Gamma, \ell^\infty(\Gamma))$ is
isomorphic to the uniformly finite homology by Block and Weinberger~\cite{block-weinberger}.
But unlike for $H^\ast(\Gamma, \ell^\infty(\Gamma))$
the Hilbert-space structure allows to numerically measure the size
of the groups $H^\ast(\Gamma, \ell^2(\Gamma))$,
which are in general huge and unwieldy as abelian groups.
\subsection{Examples and computations}

Computations of $\ell^2$-Betti numbers of groups
are rare, especially the ones where there is a non-zero $\ell^2$-Betti number
in some degree.
But sometimes the computation, at least the non-vanishing result, follows quite formally. We present two such cases.

\begin{example}
	Let $\Gamma=F_2$ be the free group of rank~2. A geometric model is the $4$-regular tree~$T$.
	By the explicit cocycle construction in
	Example~\ref{exa: free group} we already know that $\betti_1(\Gamma)\ne 0$. But since $\betti_0(\Gamma)=0$ (since $\Gamma$ is
	infinite) and $\betti_p(\Gamma)=0$ for $p>1$, this also follows from the Euler-Poincare formula:
	\[ \betti_1(\Gamma)=-\chi(T;\Gamma)=-\chi(S^1\vee S^1)=1\]
\end{example}

The only relevant information for the previous example was the number of equivariant cells in the
geometric model. The same technique helps in the next example (cf.~\cite[5.29]{petersen-phd}).

\begin{example}
	Let $G=SL_3(\bbQ_p)$. For exact computations of $\ell^2$-Betti numbers of Chevalley groups over $\bbQ_p$ or $\mathbb{F}_p((t^{-1}))$ (and their lattices) we refer to~\cite{petersen+sauer+thom}. Here we only 
	show $\betti_2(G,\mu)\ne 0$ as easily as possible and then apply this to
	the deficiency of lattices in $G$.
	As geometric model, we take the Bruhat-Tits building~$X$ of~$G$, which is $2$-dimensional.
	By the fact that there are no $3$-cells and by additivity of dimension, we obtain that
	\begin{align*}
		\betti_2(G,\mu)&=\dim_{(G,\mu)}(\coker(d^1)) \\&\ge
		\dim_{(G,\mu)}\bigl(\hom_\bbC(C^\CW_2(X), L^2(G))^G \bigr)
		-\dim_{(G,\mu)}\bigl(\hom_\bbC(C^\CW_1(X), L^2(G))^G \bigr)
	\end{align*}
	In dimension $2$ there is only one equivariant cell with stabiliser~$B$, the Iwahori subgroup of $G$.
	Hence
	\[
		\hom_\bbC(C^\CW_2(X), L^2(G))^G\cong\map(G/B, L^2(G))^G\cong L^2(G)^B
	\]
	and with Example~\ref{exa: compact-open subgroup} it follows that
	$\dim_{(G,\mu)}\bigl(\hom_\bbC(C^\CW_2(X), L^2(G))^G\bigr)=1/\mu(B)$.
	We normalize $\mu$ such that $\mu(B)=1$. There are three equivariant $1$-cells corresponding to the
	$1$-dimensional faces of the $2$-dimensional fundamental chamber. The stabiliser of each
	splits into $p+1$ many cosets of $B$. Therefore the $\mu$-measure of each stabiliser is $(p+1)$.
	Similarly as above, this yields
	\[
		\dim_{(G,\mu)}\bigl(\hom_\bbC(C^\CW_1(X), L^2(G))^G\bigr)=\frac{3}{p+1}.
	\]
	Let $p\ge 3$. Then we obtain that $\betti_2(G,\mu)\ge 1-\frac{3}{p+1}>0$.
	Let us consider a lattice $\Gamma<G$. By Theorem~\ref{thm: locally compact prop principle},
	\[\betti_2(\Gamma)=\mu(\Gamma\bs G)\betti_p(G,\mu)\ge  (1-\frac{3}{p+1})\mu(\Gamma\bs G).\]
	Let $R$ be a finite presentation of $\Gamma$, and let $g$ be the number of generators and $r$ be the
	number of relations in $R$. Let $X(R)$ be the universal covering of the presentation complex of $R$.
	One can regard $X(R)$ as the $2$-skeleton of a geometric model $Y$ from which we can compute the
	$\ell^2$-Betti numbers of $\Gamma$. By the Euler-Poincare formula,
	\[ g-r=1-\chi(X(R);\Gamma)\le 1-\betti_0(\Gamma)+\betti_1(\Gamma)-\betti_2(\Gamma)\le 1-(1-\frac{3}{p+1})\mu(\Gamma\bs G).\]
	We also used that $\Gamma$ has property (T) which implies that $\betti_1(\Gamma)=0$.
	Hence the deficiency of $\Gamma$, which is defined as the maximal value $g-r$ over all finite presentations, is bounded from
	above by $1-(1-3/(p+1))\mu(\Gamma\bs G)$.
\end{example}

\begin{remark}
The computation of $\ell^2$-Betti numbers of locally compact groups reduces to the case of
totally disconnected groups. Let $G$ be a (second countable, unimodular) locally compact group.
If its amenable radical~$K$, its largest normal amenable (closed) subgroup, is non-compact, then
$\betti_p(G,\mu)=0$ for all $p\ge 0$ by~\cite[Theorem~C]{kyed+petersen+vaes}, which generalises a result of
Cheeger and Gromov for discrete groups~\cite{cheeger+gromov}. So let us assume that $K$ is compact. Endowing $G/K$ with
the pushforward $\nu$ of~$\mu$, one obtains that $\betti_p(G/K,\nu)=\betti_p(G,\mu)$~\cite[Theorem~3.14]{petersen-phd}.
So we may and will assume
that the amenable radical of $G$ is trivial. Upon replacing $G$ by a subgroup of finite index, $G$ splits then
as a product of a centerfree non-compact semisimple Lie group $H$ and a totally disconnected group $D$.
This is an observation of Burger and Monod~\cite[Theorem~11.3.4]{monod-book}, based on the positive solution of Hilbert's 5th problem.
Since $H$ possesses lattices, one can use to Borel's computations of $\ell^2$-Betti numbers of such
lattices~\cite[Chapter~5]{lueck-book} and Theorem~\ref{thm: locally compact prop principle} to obtain a computation for~$H$.
A K\"unneth formula~\cite[Theorem~6.7]{petersen-phd} then yields the $\ell^2$-Betti numbers of $G=H\times D$ provided one is able to compute the $\ell^2$-Betti numbers of $D$.
\end{remark}

\end{document}